\documentclass[reqno]{amsart}
\usepackage{amsmath}
\usepackage[dvips]{graphicx}
\usepackage{amsfonts}
\usepackage{amssymb}
\usepackage{latexsym}
\graphicspath{{Img/}}

\newtheorem{theorem}{Theorem}

\theoremstyle{plain}

\newtheorem{example}{Example}

\newtheorem{lemma}{Lemma}

\newtheorem{problem}{Problem}
\newtheorem{proposition}{Proposition}

\numberwithin{equation}{section}

\begin{document}
\title[Analytical solution of the weighted Fermat-Torricelli problem]{Analytical solution of the weighted Fermat-Torricelli problem for convex quadrilaterals in the Euclidean plane: The case of two pairs of equal weights}
\author{Anastasios N. Zachos}
\address{University of Patras, Department of Mathematics, GR-26500 Rion, Greece}
\email{azachos@gmail.com}

\begin{abstract}
The weighted Fermat-Torricelli problem for four non-collinear
points in $\mathbb{R}^{2}$ states that:

Given four non-collinear points $A_{1},$ $A_{2},$ $A_{3},$ $A_{4}$
and a positive real number (weight) $B_{i}$ which correspond to
each point $A_{i},$ for $i=1,2,3,4,$ find a fifth point such that
the sum of the weighted distances to these four points is
minimized. We present an analytical solution for the weighted
Fermat-Torricelli problem for convex quadrilaterals in
$\mathbb{R}^{2}$ for the following two cases:

(a) $B_{1}=B_{2}$ and $B_{3}=B_{4},$ for $B_{1}>B_{4}$ and (b)
$B_{1}=B_{3}$ and $B_{2}=B_{4}.$

\end{abstract}\maketitle

\section{Introduction}

The weighted Fermat-Torricelli problem for $n$ non-collinear
points in $\mathbb{R}^{2}$ refers to finding the unique point
$A_{0}\in \mathbb{R}^{2},$ minimizing the objective function:

\[f(X)=\sum_{i=1}^{n}B_{i}\|X-A_{i}\|,\] $X\in \mathbb{R}^{2}$
given four non-collinear points
$\{A_{1},A_{2},A_{3},A_{4},...,A_{n}\}$ with corresponding
positive real numbers (weights) $B_{1}, B_{2}, B_{3},
B_{4},...,B_{n}$ where $\|\cdot\|$ denotes the Euclidean distance.

The existence and uniqueness of the weighted Fermat-Torricelli
point and a complete characterization of the solution of the
weighted Fermat-Torricelli problem has been given by Y. S Kupitz
and H. Martini (see \cite{Kup/Mar:97}, theorem 1.1, reformulation
1.2 page 58, theorem 8.5 page 76, 77). A particular case of this
result for four non-collinear points in $\mathbb{R}^{2},$ is given
by the following theorem:

\begin{theorem}{\cite{BolMa/So:99},\cite{Kup/Mar:97}}\label{theor1}
Let there be given four non-collinear points
$\{A_{1},A_{2},A_{3},A_{4}\},$ $A_{1}, A_{2},
A_{3},A_{4}\in\mathbb{R}^{2}$ with corresponding positive
weights $B_{1}, B_{2}, B_{3}, B_{4}.$ \\
(a) The weighted Fermat-Torricelli point $A_{0}$ exists and is
unique. \\
(b) If for each point $A_{i}\in\{A_{1},A_{2},A_{3},A_{4}\}$

\begin{equation}\label{floatingcase}
\|{\sum_{j=1, i\ne j}^{4}B_{j}\vec u(A_i,A_j)}\|>B_i,
\end{equation}

 for $i,j=1,2,3$  holds,
 then \\
 ($b_{1}$) the weighted Fermat-Torricelli point $A_{0}$ (weighted floating equilibrium point) does not belong to $\{A_{1},A_{2},A_{3},A_{4}\}$
 and \\
 ($b_{2}$)

\begin{equation}\label{floatingequlcond}
 \sum_{i=1}^{4}B_{i}\vec u(A_0,A_i)=\vec 0,
\end{equation}
where $\vec u(A_{k} ,A_{l})$ is the unit vector from $A_{k}$ to
$A_{l},$ for $k,l\in\{0,1,2,3,4\}$
 (Weighted Floating Case).\\
 (c) If there is a point $A_{i}\in\{A_{1},A_{2},A_{3},A_{4}\}$
 satisfying
 \begin{equation}
 \|{\sum_{j=1,i\ne j}^{4}B_{j}\vec u(A_i,A_j)}\|\le B_i,
\end{equation}
then the weighted Fermat-Torricelli point $A_{0}$ (weighted
absorbed point) coincides with the point $A_{i}$ (Weighted
Absorbed Case).
\end{theorem}

In 1969, E. Cockayne, Z. Melzak proved in \cite{CockayneMelzak:69}
by using Galois theory that for a specific set of five
non-collinear points the unweighted Fermat-Torricelli point
$A_{0}$ cannot be constructed by ruler and compass in a finite
number of steps (Euclidean construction).

In 1988, C. Bajaj also proved in \cite{Bajaj:87} by applying
Galois theory that for $n\ge 5$ the weighted Fermat-Torricelli
problem for $n$ non-collinear points is in general not solvable by
radicals over the field of rationals in $\mathbb{R}^{3}.$

We recall that for $n=4,$ Fagnano proved that the solution of the
unweighted  Fermat-Torricelli problem ($B_{1}=B_{2}=B_{3}=B_{4}$)
for convex quadrilaterals in $\mathbb{R}^{2}$ is the intersection
point of the two diagonals and it is well known that the solution
of the weighted Fermat-Torricelli problem for non-convex
quadrilaterals is the vertex of the non-convex angle. Extensions
of Fagnano result to some metric spaces are given by Plastria in
\cite{Plastria:06}.

In 2012,  Roussos studied the unweighted Fermat-Torricelli problem
for Euclidean triangles and Uteshev studied the corresponding
weighted Fermat-Torricelli problem and succeeded in finding an
analytic solution by using some algebraic system of equations (see
\cite{Roussos:12} and \cite{Uteshev:12}).

Thus, we consider the following open problem:

\begin{problem}
Find an analytic solution with respect to the weighted
Fermat-Torricelli problem for convex quadrilaterals in
$\mathbb{R}^{2},$ such that the corresponding weighted
Fermat-Torricelli point is not any of the given points.
\end{problem}

In this paper, we present an analytic solution for the weighted
Fermat-Torricelli problem for a given tetragon in $\mathbb{R}^{2}$
for $B_{1}>B_{4},$ $B_{1}=B_{2}$ and $B_{3}=B_{4},$ by expressing
the objective function as a function of the linear segment which
connects the intersection point of the two diagonals and the
corresponding weighted Fermat-Torricelli point (Section~2,
Theorem~\ref{theortetr}).

By expressing the angles $\angle A_{1}A_{0}A_{2}$ $\angle
A_{2}A_{0}A_{3},$ $\angle A_{3}A_{0}A_{4}$ and $\angle
A_{4}A_{0}A_{1}$ as a function of $B_{1},$ $B_{4}$ and $a$ and
taking into account the invariance property of the weighted
Fermat-Torricelli point, we obtain an analytic solution for a
convex quadrilateral having the same weights with the tetragon
(Section~3, Theorem~\ref{theorquadnn}).

Finally, we derive that the solution for the weighted
Fermat-Torricelli problem for a given convex quadrilateral in
$\mathbb{R}^{2}$ for the weighted floating case for $B_{1}=B_{3}$
and $B_{2}=B_{4}$ is the intersection point (Weighted
Fermat-Torricelli point) of the two diagonals (Section~4,
Theorem~\ref{diagonalquad}).

\section{The weighted Fermat-Torricelli problem for a tetragon: The case  $B_{1}=B_{2}$ and $B_{3}=B_{4}.$ }

We consider the weighted Fermat-Torricelli problem for a tetragon
$A_{1}A_{2}A_{3}A_{4},$ for $B_{1}>B_{4},$ $B_{1}=B_{2}$ and
$B_{3}=B_{4}.$

We denote by $a_{ij}$ the length of the linear segment $A_iA_j,$
$O$ the intersection point of $A_{1}A_{3}$ and $A_{2}A_{4},$ $y$
the length of the linear segment $OA_{0}$ and $\alpha_{ikj}$ the
angle $\angle A_{i}A_{k}A_{j}$ for $i,j,k=0,1,2,3,4, i\neq j\neq
k$ (See fig.~\ref{fig1}) and we set
$a_{12}=a_{23}=a_{34}=a_{41}=a.$

\begin{figure}
\centering
\includegraphics[scale=0.2]{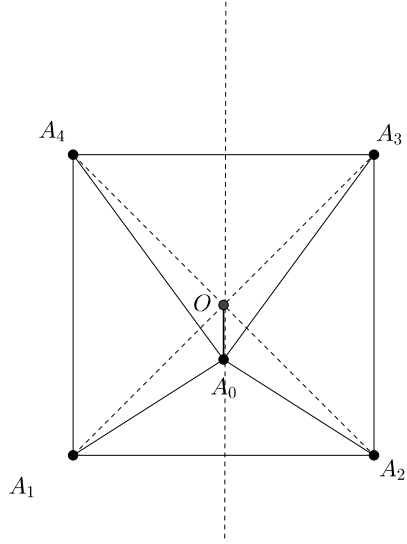}
\caption{The weighted Fermat-Torricelli problem for a tetragon
$B_{1}=B_{2}$ and $B_{3}=B_{4}$ for $B_{1}>B_{4}$}\label{fig1}
\end{figure}

\begin{problem}\label{sym1}
Given a tetragon $A_{1}A_{2}A_{3}A_{4}$ and a weight $B_{i}$ which
corresponds to the vertex $A_{i},$ for $i=1,2,3,4,$ find a fifth
point $A_{0}$ (weighted Fermat-Torricelli point) which minimizes
the objective function

\begin{equation}\label{obj1}
f=B_{1}a_{01}+B_{2} a_{02}+ B_{3} a_{03}+B_{4} a_{04}
\end{equation}
for $B_{1}>B_{4},$ $B_{1}=B_{2}$ and $B_{3}=B_{4}.$
\end{problem}

\begin{theorem}\label{theortetr}
The location of the weighted Fermat-Torricelli point of
$A_{1}A_{2}A_{3}A_{4}$ for $B_{1}=B_{2},$ $B_{3}=B_{4}$ and
$B_{1}>B_{4}$ is given by:

\begin{eqnarray}\label{analsoltetragon}
&&y=\frac{1}{2} \sqrt{\frac{a^2}{4}+r}-\nonumber{}\\
&&{}-\frac{1}{2} \sqrt{\frac{a^2}{4}-\frac{t^{1/3}}{24\ 2^{1/3}
q^{1/3}}-\frac{25 p q^{1/3}}{3\ 2^{2/3} t^{1/3}
\left(B_1^2-B_4^2\right){}^2} +\frac{a^2 B_1^2-a^2B_4^2}{12
\left(B_1^2-B_4^2\right)}-\frac{-a^3 B_1^2-a^3 B_4^2}{2
\sqrt{\frac{a^2}{4}+r} \left(B_1^2-B_4^2\right)}}\nonumber{}\\
\end{eqnarray}

where

\begin{eqnarray}\label{analsoltetragon1}
&&t=2000 a^6 B_1^6-2544 a^6 B_1^4 B_4^2+2544 a^6 B_1^2 B_4^4-2000
a^6 B_4^6+\nonumber\\
&&{}+192 \sqrt{3} \sqrt{a^{12} B_1^2 B_4^2
\left(B_1^2-B_4^2\right){}^2 \left(125 B_1^4-142 B_1^2 B_4^2+125
B_4^4\right)},
\end{eqnarray}

\begin{eqnarray}\label{analsoltetragon2}
p=a^4 B_1^4-2 a^4 B_1^2 B_4^2+a^4 B_4^4,
\end{eqnarray}

\begin{eqnarray}\label{analsoltetragon3}
q=B_1^6-3 B_1^4 B_4^2+3 B_1^2 B_4^4-B_4^6
\end{eqnarray}
and
\begin{eqnarray}\label{analsoltetragon4}
r=\frac{t^{1/3}}{24\ 2^{1/3} q^{1/3}}+\frac{25 p q^{1/3}}{3\
2^{2/3} t^{1/3} \left(B_1^2-B_4^2\right){}^2}-\frac{a^2 B_1^2-a^2
B_4^2}{12 \left(B_1^2-B_4^2\right)}.
\end{eqnarray}

\end{theorem}

\begin{proof}[Proof of Theorem~\ref{theortetr}:]

Taking into account the symmetry of the weights $B_{1}=B_{4}$ and
$B_{2}=B_{3}$ for $B_{1}>B_{4}$ and the symmetries of the tetragon
the objective function (\ref{obj1}) of the weighted
Fermat-Torricelli problem (Problem~\ref{sym1})  could be reduced
to an equivalent Problem by placing a wall to the midperpendicular
line from $A_{1}A_{2}$ and $A_{3}A_{4}$ which states that: Find a
point $A_{0}$ which belongs to the midperpendicular of
$A_{1}A_{2}$ and $A_{3}A_{4}$ and minimizes the objective function

\begin{equation}\label{obj12}
\frac{f}{2}=B_{1}a_{01}+B_{4} a_{04}.
\end{equation}

\begin{figure}
\centering
\includegraphics[scale=0.2]{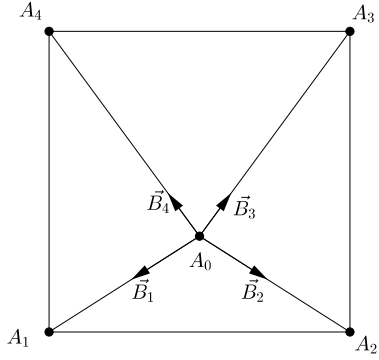}
\caption{The weighted floating equilibrium point (weighted
Fermat-Torricelli point) $A_{0}$ for a tetragon $B_{1}=B_{2}$ and
$B_{3}=B_{4}$ for $B_{1}>B_{4}$}\label{fig2}
\end{figure}

We express $a_{01},$ $a_{02},$ $a_{03}$ and $a_{04}$ as a function
of $y:$

\begin{equation}\label{a01}
a_{01}^{2}=\left(\frac{a}{2}\right)^{2}+\left(\frac{a}{2}-y\right)^{2}
\end{equation}

\begin{equation}\label{a02}
a_{02}^{2}=\left(\frac{a}{2}\right)^{2}+\left(\frac{a}{2}-y\right)^{2}
\end{equation}

\begin{equation}\label{a03}
a_{03}^{2}=\left(\frac{a}{2}\right)^{2}+\left(\frac{a}{2}+y\right)^{2}
\end{equation}

\begin{equation}\label{a04}
a_{04}^{2}=\left(\frac{a}{2}\right)^{2}+\left(\frac{a}{2}+y\right)^{2}
\end{equation}

By replacing (\ref{a01}) and (\ref{a04}) in (\ref{obj12}) we get:

\begin{equation}\label{obj13}
B_{1}\sqrt{\left(\frac{a}{2}\right)^{2}+\left(\frac{a}{2}-y\right)^{2}}+B_{4}\sqrt{\left(\frac{a}{2}\right)^{2}+\left(\frac{a}{2}+y\right)^{2}}
\to min.
\end{equation}

By differentiating (\ref{obj13}) with respect to $y,$ and by
squaring both parts of the derived equation, we get:

\begin{equation}\label{fourth1}
\frac{B_{1}^2 \left(\frac{a}{2}-y\right)^{2}
}{\left(\frac{a}{2}\right)^{2}+\left(\frac{a}{2}-y\right)^{2}}=\frac{B_{4}^2
\left(\frac{a}{2}+y\right)^{2}}{\left(\frac{a}{2}\right)^{2}+\left(\frac{a}{2}+y\right)^{2}}
\end{equation}

or

\begin{equation}\label{fourth2}
8 \left(B_{1}^{2}-B_{4}^2\right) y^4+2 a^2
\left(-B_{1}^2+B_{4}^2\right) y^2-2 a^3
\left(B_{1}^2+B_{4}^2\right) y+a^4 \left(B_{1}^2-B_{4}^2\right)=0.
\end{equation}

By solving the fourth order equation with respect to $y,$ we
derive two complex solutions and two real solutions (Ferrari's
solution, see also in \cite{Shmakov:11}) which depend on $B_{1},
B_{4}$ and $a.$ One of the two real solutions with respect to $y$
is (\ref{analsoltetragon}). From (\ref{analsoltetragon}), we
obtain that the weighted Fermat-Torricelli point $A_{0}$ is
located at the interior of $A_{1}A_{2}A_{3}A_{4}$ (see
fig.~\ref{fig2}).

\end{proof}


The Complementary Fermat-Torricelli problem was stated by Courant
and Robbins (see in \cite[pp.~358]{Cour/Rob:51}) for a triangle
which is derived by the weighted Fermat-Torricelli problem by
placing one negative weight to one of the vertices of the triangle
and asks for the complementary weighted Fermat-Torricelli point
which minimizes the corresponding objective function.

We need to state the Complementary weighted Fermat-Torricelli
problem for a tetragon, in order to explain the second real
solution which have been obtained by (\ref{fourth2}) with respect
to $y.$

\begin{problem}\label{sym2complementary}
Given a tetragon $A_{1}A_{2}A_{3}A_{4}$ and a weight $B_{i}$ (a
positive or negative real number) which corresponds to the vertex
$A_{i},$ for $i=1,2,3,4,$ find a fifth point $A_{0}$ (weighted
Fermat-Torricelli point) which minimizes the objective function

\begin{equation}\label{obj1}
f=B_{1}a_{01}+B_{2} a_{02}+ B_{3} a_{03}+B_{4} a_{04}
\end{equation}
for $\|B_{1}\|>\|B_{4}\|,$ $B_{1}=B_{2}$ and $B_{3}=B_{4}.$
\end{problem}

\begin{proposition}\label{theortetrcomp1}
The location of the complementary weighted Fermat-Torricelli point
$A_{0}^{\prime}$ (solution of Problem~\ref{sym2complementary}) of
$A_{1}A_{2}A_{3}A_{4}$ for $B_{1}=B_{2}<0,$ $B_{3}=B_{4}<0$ and
$\|B_{1}\|>\|B_{4}\|$ coincides with the location of the
corresponding weighted Fermat-Torricelli point of
$A_{1}A_{2}A_{3}A_{4}$ for $B_{1}=B_{2}>0,$ $B_{3}=B_{4}>0$ and
$\|B_{1}\|>\|B_{4}\|.$

\end{proposition}

\begin{proof}[Proof of Proposition~\ref{theortetrcomp1}:]
By applying theorem~\ref{theortetr} for  $B_{1}=B_{2}<0,$
$B_{3}=B_{4}<0$ we derive the weighted floating equilibrium
condition (see fig.~\ref{fig3}):

\begin{figure}
\centering
\includegraphics[scale=0.2]{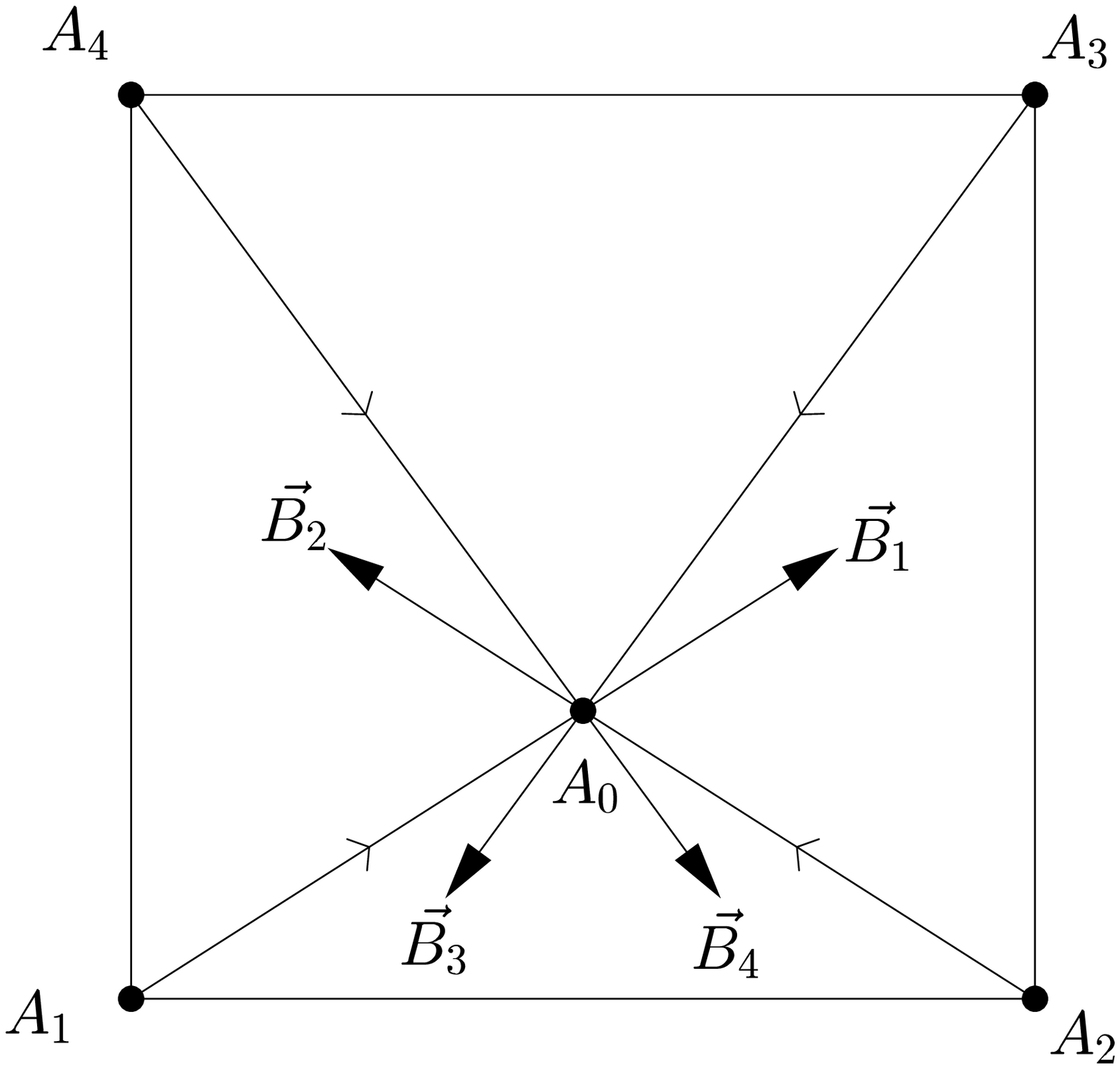}
\caption{The complementary weighted Fermat-Torricelli point
$A_{0}^{\prime}$ for a tetragon $B_{1}=B_{2}<0$ and
$B_{3}=B_{4}<0$ for $\|B_{1}\|>\|B_{4}\|$}\label{fig3}
\end{figure}

\begin{equation}\label{compl1}
\vec{B_{1}}+\vec{B_{2}}+\vec{B_{3}}+\vec{B_{4}}=\vec{0}
\end{equation}

or

\begin{equation}\label{compl2}
(-\vec{B_{1}})+(-\vec{B_{2}})+(-\vec{B_{3}})+(-\vec{B_{4}})=\vec{0}.
\end{equation}

From (\ref{compl1}) and (\ref{compl2}), we derive that the
complementary weighted Fermat-Torricelli point $A_{0}^{\prime}$
coincides with the weighted Fermat-Torricelli point $A_{0}.$ The
difference between the figures~\ref{fig2} and ~\ref{fig3} is that
the vectors $\vec{B}_{i}$  change direction from $A_{i}$ to
$A_{0},$ for $i=1,2,3,4.$

\end{proof}

\begin{proposition}\label{theortetrcomp2}
The location of the complementary weighted Fermat-Torricelli point
$A_{0}^{\prime}$ (solution of Problem~\ref{sym2complementary}) of
$A_{1}A_{2}A_{3}A_{4}$ for $B_{1}=B_{2}<0,$ $B_{3}=B_{4}>0$ or
$B_{1}=B_{2}>0,$ $B_{3}=B_{4}<0$ and  $\|B_{1}\|>\|B_{4}\|$ is
given by:

\begin{eqnarray}\label{analsoltetragoncom}
&&y=\frac{\sqrt{d}}{2}+\frac{1}{2} \nonumber\\
&&{}\sqrt{-\frac{\frac{2\ 2^{1/3}
w}{\left(\sqrt{s}+z\right)^{1/3}}+2^{2/3}
\left(\sqrt{s}+z\right)^{1/3}+32 a \left(2+a
\left(-2-\frac{3}{\sqrt{d}}\right)\right) B_1^2+32 a \left(-2+2
a-\frac{3 a}{\sqrt{d}}\right) B_4^2}{96
\left(B_1^2-B_4^2\right)}}.\nonumber\\
\end{eqnarray}

where

\begin{eqnarray}\label{analsoltetragon1com}
&&z=-1024 \left(-a B_1^2+a^2 B_1^2+a B_4^2-a^2
B_4^2\right){}^3+27648 \left(B_1^2-B_4^2\right) \left(a^2
B_1^2+a^2 B_4^2\right){}^2+\nonumber\\
&&{}+9216 \left(B_1^2-B_4^2\right) \left(-a B_1^2+a^2 B_1^2+a
B_4^2-a^2 B_4^2\right) \left(2 a^3 B_1^2+a^4 B_1^2-2 a^3 B_4^2-a^4
B_4^2\right)\nonumber\\
\end{eqnarray}

\begin{eqnarray}\label{analsoltetragon2com}
w=64 \left(-a B_1^2+a^2 B_1^2+a B_4^2-a^2 B_4^2\right){}^2+192
\left(B_1^2-B_4^2\right) \left(2 a^3 B_1^2+a^4 B_1^2-2 a^3
B_4^2-a^4 B_4^2\right),\nonumber\\
\end{eqnarray}

\begin{eqnarray}\label{analsoltetragon3com}
&&s=-4 w^3+(-1024 \left(-a B_1^2+a^2 B_1^2+a B_4^2-a^2
B_4^2\right){}^3+27648 \left(B_1^2-B_4^2\right) \left(a^2
B_1^2+a^2 B_4^2\right){}^2+\nonumber\\
&&{}+9216 \left(B_1^2-B_4^2\right) \left(-a B_1^2+a^2 B_1^2+a
B_4^2-a^2 B_4^2\right) \left(2 a^3 B_1^2+a^4 B_1^2-2 a^3 B_4^2-a^4
B_4^2\right)){}^2
\end{eqnarray}
and
\begin{eqnarray}\label{analsoltetragon4com}
&&d=\frac{1}{2} \left(-a+a^2\right)+\frac{w}{24\ 2^{2/3}
\left(\sqrt{s}+z\right)^{1/3}
\left(B_1^2-B_4^2\right)}+\frac{\left(\sqrt{s}+z\right)^{1/3}}{48\
2^{1/3} \left(B_1^2-B_4^2\right)}-\nonumber\\
&&{}-\frac{-a B_1^2+a^2 B_1^2+a B_4^2-a^2 B_4^2}{6
\left(B_1^2-B_4^2\right)}.
\end{eqnarray}

\end{proposition}

\begin{proof}[Proof of Proposition~\ref{theortetrcomp2}:]

Taking into account (\ref{obj13}) for $B_{1}=B_{2}<0,$
$B_{3}=B_{4}>0$ or $B_{1}=B_{2}>0,$ $B_{3}=B_{4}<0$ and
$\|B_{1}\|>\|B_{4}\|$ and differentiating (\ref{obj13}) with
respect to $y\equiv OA_{0}^{\prime},$ and by squaring both parts
of the derived equation, we obtain (\ref{fourth2}) which is a
fourth order equation with respect to $y.$ The second real
solution of $y$ gives (\ref{analsoltetragoncom}). From
(\ref{analsoltetragoncom}) and the vector equilibrium condition
$\vec{B_{1}}+\vec{B_{2}}+\vec{B_{3}}+\vec{B_{4}}=\vec{0}$ we
obtain that the complementary weighted Fermat-Torricelli point
$A_{0}^{\prime}$ for $B_{1}=B_{2}<0,$ $B_{3}=B_{4}>0$ coincides
with the complementary weighted Fermat-Torricelli point
$A_{0}^{\prime\prime}$ for $B_{1}=B_{2}>0,$ $B_{3}=B_{4}<0$ (
Fig.~\ref{fig4} and ~\ref{fig5}). Furthermore, the solution
(\ref{analsoltetragoncom}) yields that the complementary
$A_{0}^{\prime}$ is located outside the tetragon
$A_{1}A_{2}A_{3}A_{4}$ (Fig.~\ref{fig4} and ~\ref{fig5}).

\begin{figure}
\centering
\includegraphics[scale=0.2]{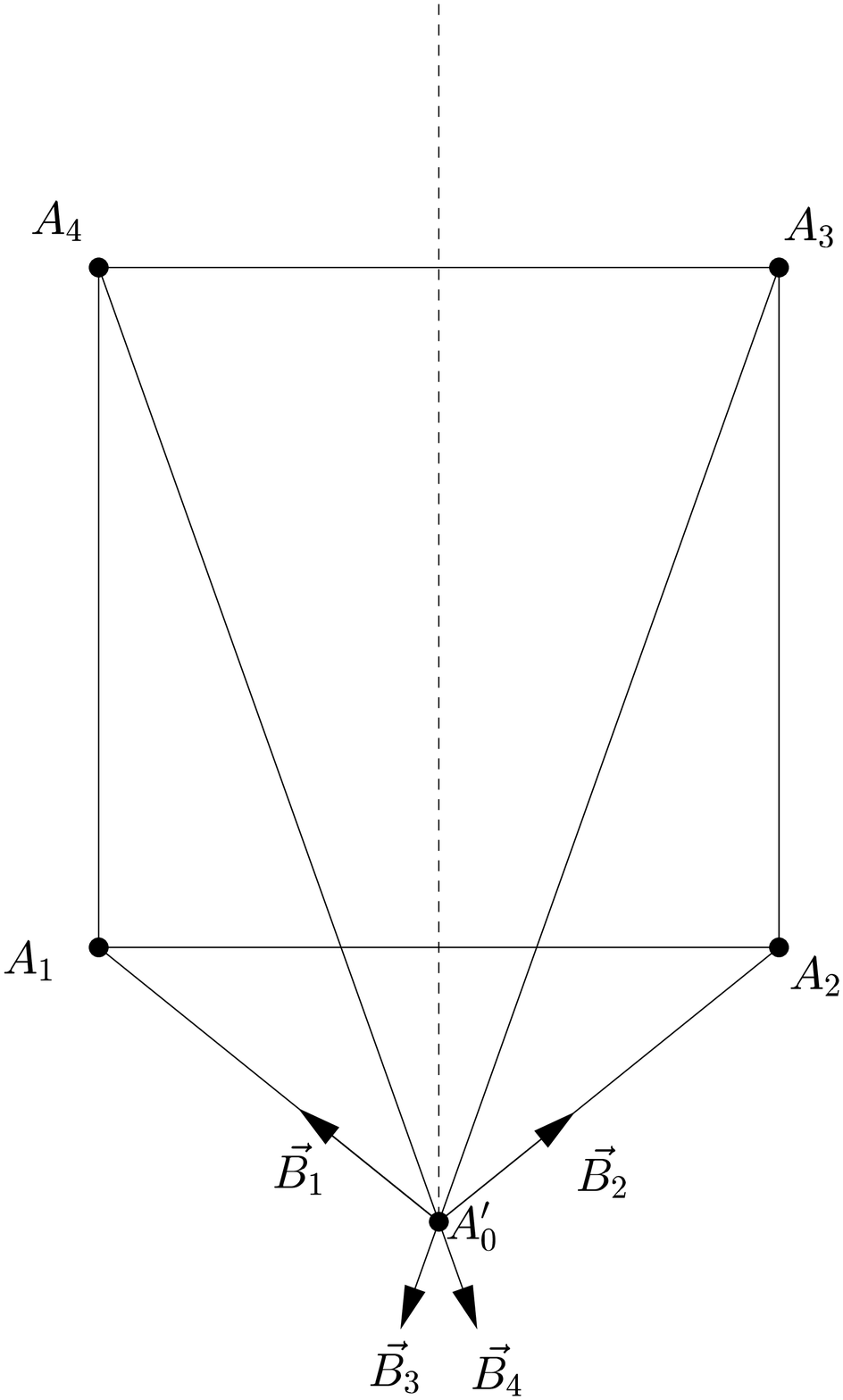}
\caption{The complementary weighted Fermat-Torricelli point
$A_{0}^{\prime}$ for a tetragon $B_{1}=B_{2}>0$ and
$B_{3}=B_{4}<0$ for $\|B_{1}\|>\|B_{4}\|$}\label{fig4}
\end{figure}

\begin{figure}
\centering
\includegraphics[scale=0.2]{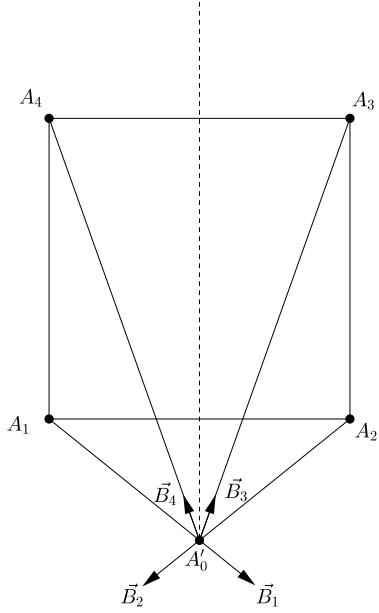}
\caption{The complementary weighted Fermat-Torricelli point for a
tetragon $B_{1}=B_{2}<0$ and $B_{3}=B_{4}>0$ for
$\|B_{1}\|>\|B_{4}\|$}\label{fig5}
\end{figure}

\end{proof}

\begin{example}\label{tetr1}
Given a tetragon $A_{1}A_{2}A_{3}A_{4}$ in $\mathbb{R}^{2},$ $a=2,
B_{1}=B_{2}=1.5,$ $B_{3}=B_{4}=1$ from \ref{analsoltetragon} and
(\ref{analsoltetragoncom}) we get $y=0.36265$ and $y=1.80699,$
respectively, with five digit precision. The weighted
Fermat-Torricelli point $A_{0}$ and the complementary weighted
Fermat-Torricelli point $A_{0^{\prime}}\equiv A_{0}$ for
$B_{1}=B_{2}=-1.5$ and $B_{3}=B_{4}=-1$ corresponds to
$y=0.36265.$ The complementary weighted Fermat-Torricelli point
$A_{0}^{\prime}$ for $B_{1}=B_{2}=-1.5$ and $B_{3}=B_{4}=1$ or
$B_{1}=B_{2}=1.5$ and $B_{3}=B_{4}=-1$ lies outside the tetragon
$A_{1}A_{2}A_{3}A_{4}$ and corresponds to $y=1.80699$
\end{example}
We denote by $A_{12}$ the intersection point of the
midperpendicular of $A_{1}A_{2}$ and $A_{3}A_{4}$ with
$A_{1}A_{2}$ and by $A_{14}$ the intersection point of the
perpendicular from $A_{0}$ to the line defined by $A_{1}A_{4}.$

We shall calculate the angles $\alpha_{102}, \alpha_{203},
\alpha_{304}$ and $\alpha_{401}.$

\begin{proposition}\label{anglestetragon}
The angles $\alpha_{102},$ $\alpha_{203},$ $\alpha_{304}$ and
$\alpha_{401}$ are given by:

\begin{equation}\label{alpha102}
\alpha_{102}=2\arccos{\frac{\frac{a}{2}-y(B_{1},B_{4},a)}{\sqrt{\left(\frac{a}{2}\right)^{2}+\left(\frac{a}{2}-y\right)^{2}}}},
\end{equation}

\begin{equation}\label{alpha304}
\alpha_{304}=2\arccos{\left(\frac{B_{1}}{B_{4}}\cos\frac{\alpha_{102}}{2}\right)}
\end{equation}

and

\begin{equation}\label{alpha401}
\alpha_{401}=\alpha_{203}=\pi-\frac{\alpha_{102}}{2}-\frac{\alpha_{304}}{2}.
\end{equation}

\end{proposition}

\begin{proof}[Proof of Proposition~\ref{anglestetragon}:]

From $\triangle A_{1}A_{12}A_{0}$ and taking into account
(\ref{analsoltetragon}), we get (\ref{alpha102}).

From the right angled triangles $\triangle A_{1}A_{12}A_{0},$
$\triangle A_{1}A_{14}A_{0}$ and $\triangle A_{4}A_{14}A_{0},$ we
obtain:

\begin{equation}\label{sin102bis}
a_{01}=\frac{a}{2\sin\frac{\alpha_{102}}{2}},
\end{equation}

\begin{equation}\label{sin304bis}
a_{04}=\frac{a}{2\sin\frac{\alpha_{304}}{2}},
\end{equation}

and

\begin{equation}\label{sin102304bis}
a_{01}\cos\frac{\alpha_{102}}{2}+a_{04}\cos\frac{\alpha_{304}}{2}=a,
\end{equation}

By dividing both members of (\ref{sin102304bis}) by
(\ref{sin102bis}) or (\ref{sin304bis}), we get:

\begin{equation}\label{cot102304bis}
\cot\frac{\alpha_{102}}{2}=2-\cot\frac{\alpha_{304}}{2}.
\end{equation}

From (\ref{cot102304bis}) the angle $\alpha_{102}$ is expressed as
a function of $\alpha_{304}:$
$\alpha_{102}=\alpha_{102}(\alpha_{304}).$

By replacing (\ref{sin102bis}) and (\ref{sin304bis}) in
(\ref{obj12}) we get:

\begin{equation}\label{objbis}
\frac{B_{1}}{\sin\frac{\alpha_{102}}{2}}+\frac{B_{4}}{\sin\frac{\alpha_{304}}{2}}\to
min.
\end{equation}

By differentiating (\ref{cot102304bis}) with respect to
$\alpha_{304},$ we derive:

\begin{equation}\label{dercot102304bis}
\frac{d\alpha_{102}}{d\alpha_{304}}=-\frac{\sin^{2}\frac{\alpha_{102}}{2}}{\sin^{2}\frac{\alpha_{304}}{2}}.
\end{equation}

By differentiating (\ref{objbis}) with respect to $\alpha_{304}$
and replacing in the derived equation (\ref{dercot102304bis}) we
obtain (\ref{alpha304}).

From the equality of triangles $\triangle A_{1}A_{0}A_{4}$ and
$A_{2}A_{0}A_{3},$ we get $\alpha_{401}=\alpha_{203}$ which yields
(\ref{alpha401}).

\end{proof}

\section{The weighted Fermat-Torricelli problem for convex quadrilaterals: The case  $B_{1}=B_{2}$ and $B_{3}=B_{4}.$ }

We need the following lemma, in order to find the weighted
Fermat-Torricelli point for a given convex quadrilateral
$A_{1}^{\prime}A_{2}^{\prime}A_{3}^{\prime}A_{4}^{\prime}$ in
$\mathbb{R}^{2},$ which has been proved in
\cite[Proposition~3.1,pp.~414]{Zachos/Zou:88} for convex polygons
in $\mathbb{R}^{2}.$

\begin{lemma}{\cite[Proposition~3.1,pp.~414]{Zachos/Zou:88} }\label{tetragonnn}
Let $A_1A_2A_{3}A_4$ be a tetragon in $\mathbb{R}^{2}$ and each
vertex $A_{i}$ has a non-negative weight $B_{i}$ for $i=1,2,3,4.$
Assume that the floating case of the weighted Fermat-Torricelli
point $A_{0}$ is valid:
\begin{equation}\label{floatingcasetetr1}
\|{\sum_{j=1, i\ne j}^{4}B_{j}\vec u(A_i,A_j)}\|>B_i.
\end{equation}
If $A_0$ is connected with every vertex $A_i$ for $i=1,2,3,4$ and
a point $A_{i}^{\prime}$ is selected with corresponding
non-negative weight $B_{i}$ on the ray that is defined by the line
segment $A_0A_i$ and the convex quadrilateral
$A_{1}^{\prime}A_{2}^{\prime}A_{3}^{\prime}A_{4}^{\prime}$ is
constructed such that:

\begin{equation}\label{floatingcasequad2}
\|{\sum_{j=1, i\ne j}^{4}B_{j}\vec
u(A_{i}^{\prime},A_{j}^{\prime})}\|>B_i,
\end{equation}
then the weighted Fermat-Torricelli point $A_{0}^{\prime}$ of
$A_{1}^{\prime}A_{2}^{\prime}A_{3}^{\prime}A_{4}^{\prime}$ is
identical with $A_{0}.$
\end{lemma}

Let $A_{1}^{\prime}A_{2}^{\prime}A_{3}^{\prime}A_{4}^{\prime}$ be
a convex quadrilateral with corresponding non-negative weights
$B_{1}=B_{2}$ at the vertices $A_{1}^{\prime}, A_{2}^{\prime}$ and
$B_{3}=B_{4}$ at the vertices $A_{3}^{\prime}, A_{4}^{\prime}.$

We select $B_{1}$ and $B_{4}$ which satisfy the inequalities
(\ref{floatingcasetetr1}), (\ref{floatingcasequad2}) and
$B_{1}>B_{4},$ which correspond to the weighted floating case of
the tetragon $A_{1}A_{2}A_{3}A_{4}$ and
$A_{1}^{\prime}A_{2}^{\prime}A_{3}^{\prime}A_{4}^{\prime}.$
Furthermore, we assume that $A_{0}$ is located at the interior of
$\triangle A_{1}^{\prime}A_{2}^{\prime}A_{3}^{\prime}.$

We denote by $a_{ij}^{\prime}$ the length of the linear segment
$A_{i}^{\prime}A_{j}^{\prime},$  $\alpha_{ikj}^{\prime}$ the angle
$\angle A_{i}^{\prime}A_{k}^{\prime}A_{j}^{\prime}$ for
$i,j,k=0,1,2,3,4, i\neq j\neq k$ (See fig.~\ref{fig6})
\begin{figure}
\centering
\includegraphics[scale=0.2]{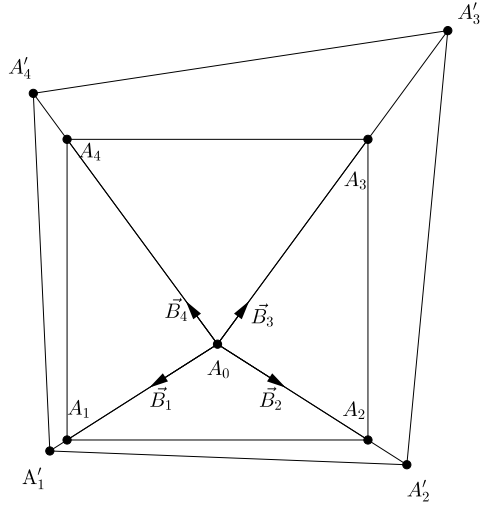}
\caption{The weighted Fermat-Torricelli point of a convex
quadrilateral for $B_{1}=B_{2},$ $B_{3}=B_{4}$ and
$B_{1}>B_{4}$}\label{fig6}
\end{figure}

\begin{theorem}\label{theorquadnn}
The location of the weighted Fermat-Torricelli point $A_{0}$ of
$A_{1}^{\prime}A_{2}^{\prime}A_{3}^{\prime}A_{4}^{\prime}$ for
$B_{1}=B_{2}$ and $B_{3}=B_{4}$ under the conditions
(\ref{floatingcasetetr1}), (\ref{floatingcasequad2}) and
$B_{1}>B_{4},$ is given by:

\begin{equation}\label{a02prime}
a_{02}^{\prime}=a_{12}^{\prime}\frac{\sin(\alpha_{213}^{\prime}-\alpha_{013}^{\prime})}{\sin\alpha_{102}}
\end{equation}

and

\begin{equation}\label{alpha120prime}
\alpha_{120}^{\prime}=\pi-\alpha_{102}-(\alpha_{123}^{\prime}-\alpha_{013}^{\prime}),
\end{equation}

where

\begin{equation} \label{eq:evquad3}
\alpha_{013}^{\prime}=\frac{\sin(\alpha_{213}^{\prime})-\cos(\alpha_{213}^{\prime})
\cot(\alpha_{102})- \frac{a_{31}^{\prime}}{a_{12}^{\prime}
}\cot(\alpha_{304}+\alpha_{401})}
{-\cos(\alpha_{213}^{\prime})-\sin(\alpha_{213}^{\prime})
\cot(a_{102})+ \frac{a_{31}^{\prime}}{a_{12}^{\prime}}}
\end{equation}

and

\begin{equation} \label{eq:evquad1}
\cot(\alpha_{304}+\alpha_{401})=
\frac{B_1+B_2\cos(\alpha_{102})+B_4\cos(\alpha_{401}))}{B_4\sin(\alpha_{401})-B_2\sin(\alpha_{102})}.
\end{equation}

\end{theorem}

\begin{proof}[Proof of Theorem~\ref{theorquadnn}:]

From lemma~\ref{tetragonnn} the weighted Fermat Torricelli point
$A_{0}$ of $A_{1}A_{2}A_{3}A_{4}$ is the same with the weighted
Fermat-Torricelli point $A_{0}^{\prime}\equiv A_{0}$ of
$A_{1}^{\prime}A_{2}^{\prime}A_{3}^{\prime}A_{4}^{\prime},$ for
the weights $B_{1}=B_{2}$ and $B_{3}=B_{4},$ under the conditions
(\ref{floatingcasetetr1}), (\ref{floatingcasequad2}) and
$B_{1}>B_{4}.$

Thus, we derive that:

$\alpha_{102}=\alpha_{102}^{\prime},$
$\alpha_{203}=\alpha_{203}^{\prime},$
$\alpha_{304}=\alpha_{304}^{\prime}$ and
$\alpha_{401}=\alpha_{401}^{\prime}.$

By applying the same technique that was used in
\cite[Solution~2.2,pp.~412-414]{Zachos/Zou:88} we express
$a_{02}^{\prime},$ $a_{03}^{\prime},$ $a_{04}^{\prime}$ as a
function of $a_{01}^{\prime}$ and $\alpha_{013}^{\prime}$ taking
into account the cosine law to the corresponding triangles
$\triangle A_{2}^{\prime}A_{1}^{\prime}A_{0}^{\prime},$ $\triangle
A_{3}^{\prime}A_{1}^{\prime}A_{0}^{\prime}$ and $\triangle
A_{4}^{\prime}A_{1}^{\prime}A_{0}^{\prime}.$ By differentiating
the objective function (\ref{obj1}) with respect to
$a_{01}^{\prime}$ and $\alpha_{013}^{\prime}$ and applying the
sine law in $\triangle
A_{2}^{\prime}A_{1}^{\prime}A_{0}^{\prime},$ $\triangle
A_{3}^{\prime}A_{1}^{\prime}A_{0}^{\prime}$ and $\triangle
A_{4}^{\prime}A_{1}^{\prime}A_{0}^{\prime}$ we derive
(\ref{eq:evquad1}) and solving with respect to
$\alpha_{013}^{\prime}$ we derive (\ref{eq:evquad3}). By applying
the sine law in $\triangle
A_{2}^{\prime}A_{1}^{\prime}A_{0}^{\prime},$ we get
(\ref{a02prime}).

Finally,
$\alpha_{120}^{\prime}=\pi-\alpha_{102}-(\alpha_{123}^{\prime}-\alpha_{013}^{\prime}).$

\end{proof}

\section{The weighted Fermat-Torricelli problem for convex quadrilaterals: The case  $B_{1}=B_{3}$ and $B_{2}=B_{4}.$ }

Let $A_{1}^{\prime}A_{2}^{\prime}A_{3}^{\prime}A_{4}^{\prime}$ be
a convex quadrilateral with corresponding non-negative weights
$B_{1}=B_{3}$ at the vertices $A_{1}^{\prime}, A_{2}^{\prime}$ and
$B_{2}=B_{4}$ at the vertices $A_{3}^{\prime}, A_{4}^{\prime}.$

We select $B_{1}$ and $B_{4}$ which satisfy the inequalities
(\ref{floatingcasetetr1}), such that $A_{0}$ is an interior point
of $A_{1}^{\prime}A_{2}^{\prime}A_{3}^{\prime}A_{4}^{\prime}.$

\begin{theorem}\label{diagonalquad}
The location of the weighted Fermat-Torricelli point $A_{0}$ of
$A_{1}^{\prime}A_{2}^{\prime}A_{3}^{\prime}A_{4}^{\prime}$ for
$B_{1}=B_{3}$ and $B_{2}=B_{4}$ under the conditions
(\ref{floatingcasetetr1}), (\ref{floatingcasequad2}) is the
intersection point of the diagonals $A_{1}^{\prime}A_{3}^{\prime}$
and  $A_{2}^{\prime}A_{4}^{\prime}.$
\end{theorem}

\begin{proof}[Proof of Theorem~\ref{diagonalquad}:]
From the weighted floating equilibrium condition
(\ref{floatingequlcond}) of theorem~\ref{theor1} we get:

\begin{equation}\label{fltcond1}
\vec{B_{1}}+\vec{B_{2}}=-(\vec{B_{3}}+\vec{B_{4}})
\end{equation}

and

\begin{equation}\label{fltcond2}
\vec{B_{1}}+\vec{B_{4}}=-(\vec{B_{2}}+\vec{B_{3}})
\end{equation}

Taking the inner product of the first part of (\ref{fltcond1})
with $\vec{B_{1}}+\vec{B_{2}}$ and the second part of
(\ref{fltcond1}) with $-(\vec{B_{3}}+\vec{B_{4}}),$ we derive
that: \[\alpha_{102}=\alpha_{304}.\]

Similarly, taking the inner product of the first part of
(\ref{fltcond2}) with $\vec{B_{1}}+\vec{B_{4}}$ and the second
part of (\ref{fltcond2}) with $-(\vec{B_{2}}+\vec{B_{3}}),$ we
derive that: \[\alpha_{104}=\alpha_{203}.\]

\end{proof}

\begin{proposition}\label{diagonalquadcom}
The location of the complementary weighted Fermat-Torricelli point
$A_{0}$ of
$A_{1}^{\prime}A_{2}^{\prime}A_{3}^{\prime}A_{4}^{\prime}$ for
$B_{1}=B_{3}<0$ and $B_{2}=B_{4}<0$ under the conditions
(\ref{floatingcasetetr1}), (\ref{floatingcasequad2}) is the
intersection point of the diagonals $A_{1}^{\prime}A_{3}^{\prime}$
and  $A_{2}^{\prime}A_{4}^{\prime}.$
\end{proposition}

\begin{proof}[Proof of Proposition~\ref{diagonalquadcom}:]
Taking into account (\ref{obj1}) for $B_{1}=B_{3}<0,$
$B_{2}=B_{4}<0$ we derive the same vector equilibrium condition
$\vec{B_{1}}+\vec{B_{2}}+\vec{B_{3}}+\vec{B_{4}}=\vec{0}.$
Therefore, we obtain that the complementary weighted
Fermat-Torricelli point $A_{0}^{\prime}$ for $B_{1}=B_{3}<0,$
$B_{2}=B_{4}<0$ coincides with the weighted Fermat-Torricelli
point $A_{0}$ of
$A_{1}^{\prime}A_{2}^{\prime}A_{3}^{\prime}A_{4}^{\prime}$  for
$B_{1}=B_{3}>0,$ $B_{2}=B_{4}>0.$
\end{proof}

The author acknowledges Professor Dr. Vassilios G. Papageorgiou
for many fruitful discussions and for his valuable comments.

\end{document}